\documentclass[11pt]{article}

\usepackage{amsmath}
\usepackage{amsthm,amscd,amsfonts}
\usepackage{amssymb, upref, color}
\usepackage{amsmath,amssymb,amsthm}
\usepackage{color}
\usepackage[colorlinks]{hyperref}
\usepackage[dvips]{graphicx}
\usepackage{latexsym,bm}
\usepackage{graphicx}
\usepackage{graphics}
\usepackage{float}

\theoremstyle{plain}
\newtheorem{exam}{Example}[section]
\newtheorem{theorem}[exam]{Theorem}
\newtheorem{lemma}[exam]{Lemma}
\newtheorem{remark}[exam]{Remark}

\begin{document}
\date{}
\title{  Solitary wave solutions of the delayed KP-BBM equation
}
\author{Yonghui Xia$^1$,\,\,\, Haojie Zhang$^1$\,\,\,and Hang Zheng$^2 $\footnote{  Corresponding author: Hang Zheng.  Email: zhenghang513@zjnu.edu.cn; zhenghwyxy@163.com.}\,\,\,\,  }
\maketitle
\noindent $^1$Department of Mathematics, Zhejiang Normal University,
Jinhua,
321004, China \\
$^2$ Department of Mathematics and Computer, Wuyi University, Wuyishan, 354300, China\\


\begin{center}
\begin{minipage}{140mm}
{\bf Abstract:} In this paper, we consider a kind of shallow water wave model called the Kadomtsev–Petviashvili–Benjamin–Bona–Mahony (KP-BBM) equation. We firstly consider the unperturbed KP-BBM equation. Then through the geometric singular perturbation (GSP) theory, especially the invariant manifold theory, method of dynamical system and Melnikov function,  we prove the persistence of solitary wave solutions for the delayed KP-BBM equation. In other words, we dissuss the equation under different nonlinear terms.  Finally, we validate our results with numerical simulations. \\
{\bf Keywords:}\ \ Kadomtsev–Petviashvili–Benjamin–Bona–Mahony equation, shallow water model, dealy, geometric singular perturbation theory

\end{minipage}
\end{center}

\section{\bf Introduction} As a kind  of nonlinear evolution equation, the water wave equations have attracted a lot of scholars’ attentions. Many famous of them have been extensively studied, including the  Korteweg-de Vries (KdV) equation, Benjamin-Bona-Mahony (BBM) equation, Kadomtsev–Petviashvili (KP) equation, Degasperis-Procesi (DP) equation, Novikov equation and  Camassa-Holm (CH) or rotation-Camassa-Holm
equation.\\
\indent Korteweg and de Vries observed a new kind of equation called the KdV equation in 1895 \cite{ref2}
\begin{equation*}
u_t+\alpha uu_x+u_{xxx}=0,
\end{equation*}
where the nonlinear term $uu_x$ in the equation causes the waveform to steepen and the dispersion effect term $u_{xxx}$ makes the waveform  expand. The interaction between them leads to the generation of solitons, while the dispersion effects are small enough to be ignored in the lowest order approximation. In particularly,  the Kadomtsov–Petviashivilli (KP) equation \cite{ref3} and the
Zakharov–Kuznetsov (ZK) equation \cite{ref4} are  two well-known generalizations of the KdV equations, given by
\begin{equation}
	\big(u_t+\alpha uu_x+u_{xxx}\big)_x+u_{yy}=0,
	\label{equation1}
\end{equation}
and
\begin{equation*}
	u_t+\alpha uu_x+(\nabla^2u)_x=0,
	\label{3}
\end{equation*}
respectively, where $\nabla^2={\partial_x}^2+{\partial_y}^2+{\partial_z}^2$ is the isotropic Laplacian. The KP equation is a partial differential
equation used to model shallow-water waves with weakly nonlinear restoring forces introduced by Kadomtsev and Petviashivilli.
Several analytical and numerical approaches were employed to solve the KP equation, and it was widely used in different physical applications. The KP equation  not only is used as a model for surface wave evolution, but also as a model for studying the propagation of ion sound waves in isotropic media \cite{ref5}. Saut and Tzvetkov \cite{ref6} studied the local well-posedness of higher-order KP equations.  They discussed the  problems on $2^{d}$ and $3^{d}$, given by
\begin{equation}
	\left\{\begin{array}{llll}
		\big(u_t+\alpha u_{xxx}+\beta u_{xxxx}+uu_x\big)_x+u_{yy}=0,
		\\
	u(0,x,y)=\phi(x,y),
	\end{array}
	\right.\label{equation2}
\end{equation}
 and
\begin{equation}
	\left\{\begin{array}{llll}
		\big(u_t+\alpha u_{xxx}+\beta u_{xxxx}+uu_x\big)_x+u_{yy}+u_{zz}=0,
		\\
		u(0,x,y)=\phi(x,y).
	\end{array}
	\right.\label{equation3}
\end{equation}
When $\beta=0$, the KP equation is divided into KP-I and KP-II \cite{ref7,ref8,ref9} based on different values of $\alpha$.\\
\indent The BBM equation, also known as the Regular Long Wave (RLW) equation, is another well-known shallow water wave equation which was investigated by Benjamin et al. \cite{ref1} for the first time
\begin{equation}
	u_t+u_x+uu_x-u_{xxt}=0.
	\label{equation4}
\end{equation}
It is an one-dimensional, unidirectional small amplitude long wave model that propagates in nonlinear dispersive media.
  Both KdV and BBM equations were
 applied to the following situations: long wave surface waves in liquids, acoustic waves in anharmonic crystals, and magnetic fluid waves in cold plasmas \cite{ref11}.\\
\indent  In this work, we study the variation problem of BBM in the KP sense, given by
\begin{equation}
(u_t+u_x-a(u^2)_x-bu_{xxt})_x+ku_{yy}=0,
	\label{equation5}
\end{equation}
where $a,b,k$ are constants. Wazwaz \cite{ref22,ref23} obtained some periodic solutions for the equation.
 Abdou \cite{ref24}  studied some exact solutions of equation ($\ref{equation5}$). Moreover, Tang et al. \cite{ref17} considered the  solitary wave solutions and periodic wave solutions of the generalized KP–BBM equation.\\
\indent Solitary  and kink wave solutions correspond to homoclinic and heteroclinic orbits in ordinary differential equations, respectively. Solitary wave solutions as a special kind of traveling wave solutions was frstly observed  by Russell in 1844 \cite{ref12}. It is well known that GSP theory \cite{ref14,ref30,ref31,ref34} is a powerful tool for dealing with singular perturbation problems because it can simplify singular perturbation systems into regular perturbation systems on invariant manifolds and ensure the existence of invariant manifolds. Many researchers have used it to study different nonlinear differential equations, including perturbed Gardner equation (Wen \cite{ref16}), generalized KdV equations (Yan et al. \cite{ref35}, Zhang et al. \cite{ref36}), perturbed BBM equations (Chen et al. \cite{ref37}, Sun and Yu \cite{ref41}, Fan and Wei \cite{ref10}  ), perturbed CH equations (Du and Li \cite{ref33}, Qiu et al. \cite{ref42}), FitzHugh-Nagumo equations (Liu et al. \cite{ref28}, Shen and Zhang \cite{ref15}), generalized Keller-Segel system (Du et al.
\cite{ref27}), perturbed sine-Gordon equation (Derks et al. \cite{ref32}), Schr$\ddot{o}$dinger equations (Xu et al. \cite{ref38}, Zheng and Xia \cite{ref26}, Li and Ji \cite{ref13}),
Belousov-Zhabotinskii system (Du and Qiao \cite{ref25}), the delayed DP equation
(Cheng and Li \cite{ref18}), biological models (Chen and Zhang \cite{ref19}, Wang and Zhang \cite{ref39}, Doelman et al. \cite{ref44}), and so on.\\
\indent The following single biological population model was investigated by Britton \cite{ref20}.
\begin{equation*}
u_t=u(1+\alpha u-(1+\alpha)g**u)+\Delta u,
\end{equation*}
 where $(g**u)(x,t)=\displaystyle\int_{-\infty}^{t}\displaystyle\int_{\Omega}g(x-y,t-\tau)u(y,\tau)dyd\tau.$ This model introduces convolution of space and time for the first time, and explains it is the drift of individuals from all possible positions in the past to their current positions. In addition to the Kuramoto-Sivashinsky (KS) perturbation, this spatio-temporal convolutional form of perturbation has attracted the attentions of many scholars and often appears in different models.  For example, Du et al. \cite{ref21} considered the  perturbed CH equation by using GSP theory.
  Ge and Wu \cite{ref29} used the same method to disscuss the delayed KdV equation. Moreover, Li et al. \cite{ref43} studied the following generalized CH-KP equation with local delay and calculated the equation with a nonlinear strength of one to verify the relevant results.
  \begin{equation}
(u_t-u_{xxt}+2ku_x-2a(f*u)u^{n-1}u_x+\tau u_{xx})_x+u_{yy}=0.
	\label{equation6*}
\end{equation}
\indent  We investigate the delayed KP-BBM equation in this paper, given by
\begin{equation}
(u_t+u_x-2a(f*u)u_x-bu_{xxt}+\tau u_{xx})_x+ku_{yy}=0,
	\label{equation6}
\end{equation}
where $f$ is a kernel function and $f*u$ represents spatial-temporal convolution with distributed delay.  Moreover,  $0<\tau\ll1$ and $\tau u_{xx}$ represents disturbance of viscous term. In this article, we discuss three cases of $f*u$:

(i) without delay
\begin{equation*}
f*u(x,y,t)=u(x,y,t);
\end{equation*}

(ii) with   local perturbed delay
\begin{equation*}
f*u(x,y,t)=\int_{-\infty}^tf(t-s)u(x,y,s)ds;
\end{equation*}

(iii) with  nonlocal perturbed delay
\begin{equation*}
f*u(x,y,t)=\int_{-\infty}^t\int_{-\infty}^{+\infty}f(x-z,t-s)u(z,y,s)dzds.
\end{equation*}
Note that, KP-BBM equation (7) is different from CH-KP equation (6) considered in Li et al. \cite{ref43}. Firstly, the coefficients of $u_{xxt}$ and $u_{yy}$ in CH-KP equation (6) are the determined numbers. However, the coefficients of $u_{xxt}$ and $u_{yy}$ in KP-BBM equation (7) are the parameters $b$ and $k$. As you will see, the two parameters   play great roles in the analysis of this paper, because they yield very significant results in this paper. That is, we consider a nonlinear PDE with more parameters and obtain  some  significant conclusions related to these parameters. Secondly, Li et al. \cite{ref43} considered  CH-KP equation (6) under locay delay. Different from Li et al. \cite{ref43}, we not only discuss the persistence of homoclinic orbits of the equation under local delay, but also consider whether its solitary wave solutions still exist under nonlocal delay. Thirdly, Li et al. \cite{ref43} did not calculate the specific  expressions of the corresponding Melnikov integral functions. However, we calculate the specific expressions of the corresponding Melnikov integral functions  and  use different methods to determine the existence of zeros.
 Finally, we find some profound differences between local and nonlocal delays, and the introductions of $b$ and $k$ will bring some difficulties to our proof. \\
\indent In section 2, we discuss solitary wave solutions of the unperturbed KP-BBM  equation. In Section 3, we study delayed KP-BBM equation by combining the GSP theory and Melnikov method. In Section 4, we validate our results with numerical simulations. In Section 5, we  provide the conclusion of the article.
\section{Unperturbed KP-BBM equation.}In this part, we discuss the unperturbed KP-BBM equation and give relevant results. Considering the following equation without delay and perturbation
\begin{equation}
	(u_t+u_x-2auu_x-bu_{xxt})_x+ku_{yy}=0.
	\label{equation11}
\end{equation}

Letting $u(x,y,t)=u(x+y-ct)=\phi(\xi)$.  We obtain
 \begin{equation}
(-c\phi'+\phi'-2a\phi\phi'+bc\phi''')'+k\phi''=0,
\label{equation12}
\end{equation}
where $ '=\frac{d}{d\xi}$ and $c>0$ is wave speed. By integrating the equation twice and taking the integration constant as $0$, we have
 \begin{equation}
-c\phi+\phi-a\phi^2+bc\phi''+k\phi=0,
	\label{equation13}
\end{equation}
 which is equivalent to
 \begin{equation}
	\left\{\begin{array}{llll}
	\phi'=\psi,
		\\
		\psi'=\frac{1}{bc}[(c-k-1)\phi+a\phi^2],
	
	\end{array}
	\right.\label{equation14}
\end{equation}
with the following first integral
\begin{equation}
	H(\phi,\psi)=\frac{1}{6}\big(3\psi^2+\frac{1}{bc}3\phi^2(k+1-c)-\frac{1}{bc}2a\phi^3\big).
	\label{equation15}
\end{equation}
Denote that
\begin{equation*}
	\phi_1=0,\phi_2=\frac{k+1-c}{a},
\end{equation*}
\begin{equation*}
	H(\phi_1,0)=h_1,H(\phi_2,0)=h_2.
\end{equation*}
Obviously, system (11) has two equilibrium $O_1(\phi_1,0)$ and $O_2(\phi_2,0)$.
\begin{theorem}\label{th3.1}
	If  $0\leq k+1<c$, $a<0$ and $b>0$,  system (11) has a homoclinic orbit $\Gamma$ surrounding
	the center point $(\phi_2,0)$ to one saddle point $(\phi_1,0)$ where level curves defined by $H(\phi,\psi)=h_1$(see Fig. 1).
	\end{theorem}
\begin{proof}
	The linearized matrix of system (11) is
	\begin{equation}
		A=
		\begin{pmatrix}
			0&1\\
			\frac{c-k-1+2a\phi}{bc}&0
		\end{pmatrix}
	\end{equation}
	Thus, if $0\leq k+1<c$, $a<0$ and $b>0$,
		\begin{equation}
		J(\phi_1,0)=detA(\phi_1,0)=
		\begin{vmatrix}
			0&1\\
			\frac{c-k-1}{bc}&0
		\end{vmatrix}
		 =-\frac{c-k-1}{bc}<0, \quad TraceA_{(\phi_1,0)}=0,
	\end{equation}
	and
		\begin{equation}
		J(\phi_2,0)=detA(\phi_2,0)=
		\begin{vmatrix}
			0&1\\
			\frac{k+1-c}{bc}&0
		\end{vmatrix}
		=-\frac{k+1-c}{bc}>0, \quad TraceA_{(\phi_1,0)}=0.
	\end{equation}
	Through the above analysis, we obtain that
	$O_1(\phi_1,0)$ and $O_2(\phi_2,0)$ are saddle point and center point, respectively. Then, if the level curves defined
	by $H(\phi_1,0)=h_1$, there exists a homoclinic orbit $\Gamma$  surrounding the center
	point $O_2(\phi_2,0)$ to one saddle point $O_1(\phi_1,0)$ (see Figure 1).
\end{proof}

\begin{figure*}[htbp]
	\centering
	\includegraphics{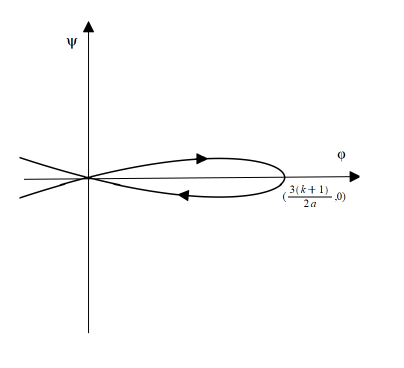}
	\caption{The homoclinic orbit within\ $0\leq k+1<c,a<0$ and $b>0$.}
	\label{fig1}
\end{figure*}

\section{Delayed KP-BBM equation.}In this part, we discuss the delayed KP-BBM equation and give relevant results.
\subsection{Model with local distributed delay}
We consider equation $(\ref{equation6})$ with local distributed delay. That is, the convolution $f*u $ will become the following form
\begin{center}
$f*u(x,y,t)=\displaystyle\int_{-\infty}^{t}f(t-s)u(x,y,s)ds,$
\end{center}
where $f(t)$ is usually represented by Gamma distributed delay kernel. The following two special cases occur frequently in differential delay equations
\begin{center}
$f(t)=\frac{1}{\tau}e^{-\frac{t}{\tau}}(n=1)$, $f(t)=\frac{4t}{\tau^2}e^{-\frac{2t}{\tau}}(n=2),$
\end{center}
where $\tau$ is called the average time delay. These two types of kernels are called weak kernels and strong kernels, respectively. Due to the similarity in the processing methods of the two types of kernel functions, we will only consider the second one in this article.\\
\indent Base on the above analysis, if $f(t)=\frac{4t}{\tau^2}e^{-\frac{2t}{\tau}}$, by letting $m=-\frac{2(t-s)}{\tau}$, we  have
\begin{equation*}
	\begin{aligned}
	&\lim\limits_{\tau \to 0}f*u(x,y,t)\\
	&=\lim\limits_{\tau \to 0}\displaystyle\int_{-\infty}^{t}f(t-s)u(x,y,s)ds\\
	     &=\lim\limits_{\tau \to 0}\displaystyle\int_{-\infty}^{t}\frac{4(t-s)}{\tau^2}e^{-\frac{2(t-s)}{\tau}}u(x,y,s)ds\\
	     &=\lim\limits_{\tau \to 0}\displaystyle\int_{-\infty}^{0}\frac{-2m\tau}{\tau^2}\cdot \frac{\tau}{2}\cdot e^mu(x,y,\frac{m\tau +2t}{2})dm\\
	     &=u(x,y,t)\displaystyle\int_{-\infty}^{0}-m \cdot e^mdm\\
	     &=u(x,y,t),
	\end{aligned}
\end{equation*}
and it is obvious that $\tau u_{xx}\rightarrow 0$ for $\tau \rightarrow 0$. The same result exists in the equation  with
nonlocal perturbation.\\
\indent In the case of  $u(x,y,t)=u(x+y-ct)=\phi(\xi)$, the following system is equivalent to equation $(\ref{equation6})$
\begin{equation}
(-c\phi'+\phi'-2a(\eta)\phi'+bc\phi'''+\tau\phi'')'+k\phi'' =0,
\label{equation19}
\end{equation}
where $0<\tau\ll1$ and
\begin{equation*}
	\begin{aligned}
		\eta(\xi)&=f*u(x,y,t)\\
		&=\displaystyle\int_{0}^{+\infty}\frac{4t}{\tau^2}e^{-\frac{2t}{\tau}}-u(x,y,t-s)dt\\
		&=\displaystyle\int_{0}^{+\infty}\frac{4t}{\tau^2}e^{-\frac{2t}{\tau}}\cdot \frac{1}{c}\cdot u(\xi)dt\\
		&=\displaystyle\int_{0}^{+\infty}\frac{4t}{\tau^2}e^{-\frac{2t}{\tau}}\cdot \frac{1}{c}\cdot \phi(\xi)dt\\
		&=\displaystyle\int_{0}^{+\infty}\frac{4t}{\tau^2}e^{-\frac{2t}{\tau}} \phi(\xi+ct)dt.
	\end{aligned}
	\label{equation11111}
\end{equation*}	
 Differentiating $\eta(\xi)$  with respect to $\xi$
 \begin{equation*}
	\begin{aligned}
		\frac{d\eta}{d\xi}&=\displaystyle\int_{0}^{+\infty}\frac{4t}{\tau^2}e^{-\frac{2t}{\tau}} \phi_{\xi}(\xi+ct)dt\\
		&=\displaystyle\int_{0}^{+\infty}\frac{4t}{\tau^2}e^{-\frac{2t}{\tau}}\cdot \frac{1}{c}\cdot \phi_{t}(\xi+ct)dt\\
		&=\frac{1}{c}\displaystyle\int_{0}^{+\infty}\frac{4t}{\tau^2}e^{-\frac{2t}{\tau}}d\phi\\
		&=\frac{1}{c}\big[\frac{2}{\tau}\displaystyle\int_{0}^{+\infty}\frac{4t}{\tau^2}e^{-\frac{2t}{\tau}} \phi(\xi+ct)dt-\frac{1}{\tau}\displaystyle\int_{0}^{+\infty}\frac{4}{\tau}e^{-\frac{2t}{\tau}} \phi(\xi+ct)dt\big]\\
		&=\frac{1}{c\tau}(2\eta-\zeta),
	\end{aligned}
	\label{equation22}
\end{equation*}	
where
\begin{center}
$\zeta=\displaystyle\int_{0}^{+\infty}\frac{4}{\tau}e^{-\frac{2t}{\tau}}\phi(\xi+ct) dt.$
\end{center}
Similarly, we have
\begin{center}
$\frac{d\zeta}{d\xi}=\frac{1}{c\tau}(2\zeta-4\phi).$
\end{center}
Thus, by integrating both sides of equation $(\ref{equation19})$ and taking the integration constant as $0$, we have
\begin{equation}
-c\phi'+\phi'-2a(\eta)\phi'+bc\phi'''+\tau\phi''+k\phi' =0.
\label{equation21}
\end{equation}
It is obvious that system $(\ref{equation21})$ reduces to the following  singularly perturbed system
\begin{equation}
\left\{\begin{array}{llll}
\phi'=\psi,
\\
\psi'=\omega,
\\
\omega'=\frac{1}{bc}[(c-k-1)\psi+2a\eta\psi-\tau \omega\big)],
\\
c\tau \eta'=2\eta-\zeta,
\\
c\tau \zeta'=2\zeta-4\phi,
\end{array}
\right.\label{equation22}
\end{equation}
where $'=\frac{d}{d\xi}$.  When $\tau \rightarrow 0$, we obtain $\eta \rightarrow \phi$ and system ($\ref{equation21}$) will be converted to unperturbed   equation ($\ref{equation12}$). When $\tau>0$, we  have the following corresponding fast system through time scale transformation
\begin{equation}
\left\{\begin{array}{llll}
\dot \phi=\tau \psi,
\\
\dot \psi=\tau \omega,
\\
\dot \omega=\frac{\tau}{bc}[(c-k-1)\psi+2a\eta\psi-\tau \omega],
\\
\dot \eta=\frac{1}{c}(2\eta-\zeta),
\\
\dot \zeta=\frac{1}{c}(2\zeta-4\phi),
\end{array}
\right.\label{equation23}
\end{equation}
where $\xi=\tau z$ and $\dot{}=\frac{d}{dz}$. Setting $\tau \rightarrow 0$ in both slow system ($\ref{equation22}$) and fast system ($\ref{equation23}$), we obtain the reduced system
\begin{equation}
\left\{\begin{array}{llll}
\phi'=\psi,
\\
\psi'=\omega,
\\
\omega'=\frac{1}{bc}[(c-k-1)\psi+2a\eta\psi],
\\
0=2\eta-\zeta,
\\
0=2\zeta-4\phi,
\end{array}
\right.\label{equation24}
\end{equation}
and the layer system
\begin{equation}
\left\{\begin{array}{llll}
\dot \phi=0,
\\
\dot \psi=0,
\\
\dot \omega=0,
\\
\dot \eta=\frac{1}{c}(2\eta-\zeta),
\\
\dot \zeta=\frac{1}{c}(2\zeta-4\phi).
\end{array}
\right.\label{equation25}
\end{equation}
Therefore, the critical manifold is
 \begin{center}
$M_0=\big \{ (\phi,\psi,\omega,\eta,\zeta)\in {\mathbb{R}}^5 : \eta=\phi, \zeta=2\phi \big \}.$
\end{center}
Then by a direct calculation, the linearization of ($\ref{equation23}$) with $\tau=0$ is
\begin{equation*}
\begin{pmatrix}
0&0&0&0&0\\
0&0&0&0&0\\
0&0&0&0&0\\
0&0&0&\frac{2}{c}&-\frac{1}{c}\\
-\frac{4}{c}&0&0&0&\frac{2}{c}
\end{pmatrix}
\end{equation*}
where the number of zero eigenvalues of the matrix is equal to the dimension of $M_0$. According to the GSP theory of Fenichel \cite{ref14,ref31}, $M_0$ is a normally hyperbolic manifold.  Therefore, if $0<\tau \ll 1$, we obtain the following three dimensional invariant manifold $M_\tau$
\begin{center}
$M_\tau=\big \{ (\phi,\psi,\omega,\eta,\zeta)\in {\mathbb{R}}^5 : \eta=\phi+g(\phi,\psi,\omega,\tau), \zeta=2\phi+h(\phi,\psi,\omega,\tau)\big \},$
\end{center}
where  $g$ and $h$ are smooth  on a compact domain. If $\tau=0$, the value of the both functions are zeros.
Now we expand both functions with respect to $\tau$
 \begin{equation}
\left\{\begin{array}{llll}
g(\phi,\psi,\omega,\tau)=\tau g_1(\phi,\psi,\omega)+\tau^2 g_2(\phi,\psi,\omega)+\cdots\ ,
\\
h(\phi,\psi,\omega,\tau)=\tau h_1(\phi,\psi,\omega)+\tau^2 h_2(\phi,\psi,\omega)+\cdots.
\end{array}
\right.\label{equation26}
\end{equation}
Base on the set representation of $M_\tau$, we  obtain
 \begin{equation}
\left\{\begin{array}{llll}
\dot \eta=\dot \phi  +\frac{\partial g}{\partial \phi} \dot\phi+\frac{\partial g}{\partial \psi} \dot\psi+\frac{\partial g}{\partial \omega} \dot\omega,
\\
\dot \zeta=2\dot\phi+\frac{\partial h}{\partial \phi} \dot\phi+\frac{\partial h}{\partial \psi} \dot\psi+\frac{\partial h}{\partial \omega} \dot\omega,
\end{array}
\right.\label{equation27}
\end{equation}
and
 \begin{equation}
\left\{\begin{array}{llll}
\dot \eta=\frac{1}{c}(2\eta-\zeta),
\\
\dot \zeta=\frac{2}{c}(\zeta-2\phi).
\end{array}
\right.\label{equation28}
\end{equation}
Substituting system ($\ref{equation26}$) into system ($\ref{equation27}$), one has
\begin{equation}
\left\{\begin{array}{llll}
\dot \eta=\tau \psi+o(\tau),
\\
\dot \zeta=2\tau\psi +o(\tau).
\end{array}
\right.\label{equation29}
\end{equation}
Similarly, substituting $M_\tau$ and system ($\ref{equation26}$) into system ($\ref{equation28}$), we get
 \begin{equation}
\left\{\begin{array}{llll}
\dot \eta=\frac{1}{c}[(2g_1-h_1)\tau+o(\tau)],
\\
\dot \zeta=\frac{2}{c}(h_1\tau+o(\tau)).
\end{array}
\right.\label{equation30}
\end{equation}
Comparing the coefficient of $\tau$ for systems ($\ref{equation29}$) and ($\ref{equation30}$), we have
\begin{center}
$g= c\tau\psi+o(\tau),h=c\tau\psi+o(\tau).$
\end{center}
 Restricted to $M_\tau$, the slow system is given by
\begin{equation}
\left\{\begin{array}{llll}
\phi'=\psi,
\\
\psi'=\omega,
\\
\omega'=\frac{1}{bc}[(c-k-1)\psi+2a\phi\psi+2ac\tau\psi^2-\tau\omega+o(\tau)].
\end{array}
\right.\label{equation31}
\end{equation}
 Obviously,  system ($\ref{equation31}$) has a line of equilibrium taking the form of
\begin{center}
	$\{\psi=0,\omega=0\}.$
\end{center}
Accoding to ($\ref{equation13}$), we have
\begin{center}
$-c\phi+\phi-a\phi^2+bc\omega+k\phi=0$.
\end{center}
Now we use variable transformation to get the normal form of  system ($\ref{equation31}$)
\begin{equation*}
\left\{\begin{array}{llll}
\widetilde \phi=\phi,
\\
\widetilde\psi=\psi,
\\
\widetilde\omega=(c-k-1)\phi+a\phi^2-bc\omega.
\end{array}
\right.\label{eq}
\end{equation*}
Thus, $\widetilde \phi,\widetilde\psi$ and $\widetilde\omega$ satisfy
\begin{equation*}
\left\{\begin{array}{llll}
\widetilde \phi'=\widetilde\psi,
\\
\widetilde\psi'=\frac{1}{bc}[(c-k-1)\widetilde\phi+a\widetilde\phi^2-\widetilde\omega],
\\
\widetilde\omega'=-2ac\tau\widetilde\psi^2+\frac{\tau}{bc}[(c-k-1)\widetilde\phi+a\widetilde\phi^2-\widetilde\omega]+o(\tau).
\end{array}
\right.\label{eq}
\end{equation*}
 We remove the superscript title for the convenience of  research.
\begin{equation}
	\left\{\begin{array}{llll}
		\ \phi'=\psi,
		\\
		\psi'=\frac{1}{bc}[(c-k-1)\phi+a\phi^2-\omega],
		\\
		\omega'=-2ac\tau\psi^2+\frac{\tau}{bc}[(c-k-1)\phi+a\phi^2-\omega]+o(\tau).
	\end{array}
	\right.\label{equation32}
\end{equation}
Similarly, system ($\ref{equation32}$) also has a line of equilibrium  $\{\psi=0,\phi=\frac{(k+1-c)-\sqrt{(c-k-1)^2+4a\omega}}{2a}\}.$

 It is not difficult to verify that $\widetilde M_0\equiv\{\psi=0,\phi=\frac{(k+1-c)-\sqrt{(c-k-1)^2+4a\omega}}{2a},|\omega|\leq\delta\}$ is  normally hyperbolic. Similar to the analysis of $M_0$, if $0<\tau\ll1$,  $\widetilde M_\tau$ persists and is $C^1O(\tau)$ close to $\widetilde M_0$.\\
\indent Obviously, if $\tau=0$, system ($\ref{equation32}$) is a Hamiltonian system. Then we investigate  the differences between the disturbed orbit $\Gamma_\tau$ and the undisturbed homoclinic
 orbit $\Gamma$. Moreover, $\Gamma$ satisfies
  \begin{equation}
	\left\{\begin{array}{llll}
	\phi'=\psi,
		\\
		\psi'=\frac{1}{bc}[(c-k-1)\phi+a\phi^2],
	
	\end{array}
	\right.\label{equation33}
\end{equation}
and
\begin{center}
$3\psi^2+\frac{\phi^2}{bc}[3(k+1-c)-2a\phi]=0$.
\end{center}
\begin{lemma}\cite{ref40}\label{le4.2}
Consider the following equations
 \begin{equation}
\left\{\begin{array}{llll}
x^{\prime}=u_0(x,y,\textbf{z})+\epsilon u_1 (x,y,\textbf{z})+O(\epsilon ^2),
\\
y^{\prime} =v_0(x,y,\textbf{z})+\epsilon v_1 (x,y,\textbf{z})+O(\epsilon ^2),
\\
\textbf{z}^{\prime} =\epsilon r (x,y,\textbf{z})+O(\epsilon ^2),
\end{array}
\right.\label{eq5}
\end{equation}
where $\textbf{z}$ is a vector, $x$ and $y$ are scalars. Let $T=(x,y),$ $f_0 =(u_0,v_0),$ $f_1 =(u_1,v_1).$\\
\indent If $f_0$ as a function of $T$ is divergence free and the equations above satisfy the saddle connection assumption, then we have
\begin{center}
$\Delta _1 (0,\textbf{z}_0)=\displaystyle\int_{-\infty}^{+\infty} v_0 u_1-u_0 v_1+(v_0 \frac{\partial{u_0}}{\partial{\textbf{z}}}-u_0 \frac{\partial{v_0}}{\partial{\textbf{z}}})\cdot \frac{\partial{\textbf{z}}}{\partial{\epsilon}}dt$,\\

\end{center}
where $\frac{\partial{\textbf{z}}}{\partial{\epsilon}}$  satisfes $(\frac{\partial{\textbf{z}}}{\partial{\epsilon}})'=r(T_0(t),\textbf{z}_0),$ and $\frac{\partial{\textbf{z}}}{\partial{\epsilon}}=0,$ at $t=0.$
\end{lemma}
\begin{theorem}
  There is a unique speed $c$, such that system ($\ref{equation32}$) has a  homoclinic orbit $\Gamma_\tau$ which connects with two equilibrium $O_1$ and $O_2$ for sufficiently small $\tau>0$. It lies  $C^1$ $O(\tau)$ close to the undisturbed homoclinic orbit $\Gamma$ $(\Gamma_\tau \rightarrow \Gamma$ in $C^1$ norm as $\tau \rightarrow 0).$ Therefore,  equation ($\ref{equation6}$) with  special local distributed delay has a solitary wave solution $\Gamma_\tau$.
\end{theorem}
\begin{proof}
Constrained to plane $\{\omega=0\}$, on a cross section where $(\phi,\psi)=(\frac{3(k+1-c)}{2a},0)$, the N$(\tau)$ signed distance between $\Gamma_\tau$ and $\Gamma$ is given by the following equation
\begin{center}
	$N(c,\tau)=\tau M+O(\tau^2)$.
\end{center}According to Lemma 3.1, we obtain
\begin{center}
	$M(\phi,\psi)=\frac{1}{bc}\displaystyle \int_{-\infty}^{+\infty}\psi \frac{ \partial \omega}{\partial \tau}d\xi$,
\end{center}
where $\phi,\psi$ are evaluated along $\Gamma$, and $\frac{ \partial \omega}{\partial \tau}$ satisfies
\begin{equation}
	\left\{\begin{array}{llll}
		\frac{d}{d\xi}(\frac{ \partial \omega}{\partial \tau})=-2ac\psi^2+\frac{1}{bc}[(c-k-1)\phi+a\phi^2-\omega],
		\\
		\frac{ \partial \omega}{\partial \tau}=0 \ at\  \xi=0.
	\end{array}
	\right.\label{equation34}
\end{equation}
According to part integral, $M(\phi,\psi)$ satisfies
\begin{equation}
	\begin{aligned}
		M(\phi,\psi)&=\frac{1}{bc}\int_{-\infty}^{+\infty}\psi \frac{ \partial \omega}{\partial \tau}d\xi\\
		&=\frac{1}{bc}\int_{-\infty}^{+\infty} \frac{ \partial \omega}{\partial \tau}d\phi\\
		&=\frac{1}{bc} \frac{ \partial \omega}{\partial \tau}\phi\big|_{-\infty}^{+\infty}-\frac{1}{bc}\int_{-\infty}^{+\infty}
		\phi \frac{d}{d\xi}(\frac{ \partial \omega}{\partial \tau})d\xi. \\
	\end{aligned}
\end{equation}
Since $\frac{d}{d\xi}(\frac{ \partial \omega}{\partial \tau})=-2ac\psi^2+\frac{1}{bc}[(c-k-1)\phi+a\phi^2-\omega]$  is bounded and exponentially small at $\pm\infty$, $\frac{ \partial \omega}{\partial \tau}$ increase at most sub-exponentially at $\pm\infty$. we get
\begin{equation}
	\begin{aligned}
		M(\phi,\psi)&=\frac{1}{bc}\int_{-\infty}^{+\infty}
		\phi [-2ac\psi^2+\psi']d\xi \\
		&=\frac{2a}{b}\int_{-\infty}^{+\infty}\phi\psi^2d\xi-\frac{1}{bc}\int_{-\infty}^{+\infty}\phi d\psi\\
		&=\frac{2a}{b}\int_{-\infty}^{+\infty}\phi\psi d\phi+\frac{1}{bc}\int_{-\infty}^{+\infty}\psi d\phi\\
		&=\frac{1}{bc}\int_{-\infty}^{+\infty}(2ac\phi+1)\psi d\phi\\
		&=\frac{2}{bc}\int_{0}^{\phi^*}(2ac\phi+1)\sqrt{\frac{2a\phi^3-3(k+1-c)\phi^2}{3bc}}d\phi\\
		&=\Delta(c),
	\end{aligned}
	\label{equation37}
\end{equation}
where $\phi^*=\frac{3(k+1-c)}{2a}$, and notice that $M(\phi,\psi)=\Delta(c)$ is continuous in $c$. If $c\rightarrow k+1$, then the first term of the integral $(2ac\phi+1)$ is always positive, resulting
the corresponding value of $M(\phi,\psi)$ is positive. If $c\rightarrow +\infty$,   then $(2ac\phi+1)$ is negative, resulting the integral function $M(\phi,\psi)$ is negative. It follows that  there exist a $c^*$ such that
\begin{center}
		$\Delta (c)=\Delta (c^*)=0, and$ $\frac{\partial \Delta}{\partial c}\big|_{c^*}\neq 0$.
	\end{center}
Therefore, for each  $0<\tau\ll1$ , by using implicit function theorem, there exists a $c=c^*(\tau)$ such that  $M(\phi,\psi)=0$. As a result, the distance function $N(c,\tau)$ has at least one zero solution for every $c>0$. That is, equation ($\ref{equation6}$) with  special local distributed delay has a solitary wave solution.
\end{proof}
\begin{remark}
The local delay term $f*u(x,y,t)$ and the perturbation term $\tau u_{xx}$ play  important roles in the existence of solitary wave solutions. Through simple calculations, the Melnikov integral function without the local delay is given by
\begin{center}
	$M(\phi,\psi)=\frac{2}{bc}\displaystyle\int_{0}^{\phi^*}\sqrt{\frac{2a\phi^3-3(k+1-c)\phi^2}{3bc}}d\phi$.
\end{center}
Notice that it keeps postive for $c>0$.  On the other hand, without disturbance of viscous term $\tau u_{xx}$, the corresponding Melnikov integral function is
\begin{center}
	$M(\phi,\psi)=\frac{2}{bc}\displaystyle\int_{0}^{\phi^*}2ac\phi\sqrt{\frac{2a\phi^3-3(k+1-c)\phi^2}{3bc}}d\phi$.
\end{center}
It keeps negative for $c>0$. We  see that the perturbation
drives $M(\phi,\psi)$ positive while the delaying effect drives $M(\phi,\psi)$ negative. Due to their combined effect, a unique solitary wave solution is generated.
\end{remark}

\subsection{Model with nonlocal distributed delay}
We discuss equation $(\ref{equation6})$ with nonlocal distributed delay which taking the form of
\begin{equation*}
f*u(x,y,t)=\int_{-\infty}^t\int_{-\infty}^{+\infty}f(x-z,t-s)u(z,y,s)dzds.
\end{equation*}
The  kernel function $f:(-\infty,+\infty)\times[0,+\infty)\rightarrow [0,+\infty)$ satisfes
\begin{center}
$\displaystyle\int_{0}^{+\infty}\displaystyle\int_{-\infty}^{+\infty}f(x,t)dxdt=1.$
\end{center}
We consider the special weak generic delay kernel $f(x,t)=\frac{1}{\sqrt{4\pi t}}e^{-\frac{x^2}{4t}}\frac{1}{\tau}e^{-\frac{t}{\tau}}$, then for  $0<\tau\ll1$, the convolution $f*u $ is given by
\begin{center}
$v(x,y,t)=f*u(x,y,t)=\displaystyle\int_{-\infty}^{t}\displaystyle\int_{-\infty}^{+\infty}\frac{1}{\sqrt4\pi(t-s)}e^{-\frac{(x-z)^2}{4(t-s)}}\frac{1}{\tau}e^{-\frac{t-s}{\tau}}u(z,y,s)dzds.$
\end{center}
By  direct calculations, we have the following equation that holds
\begin{center}
$v_t-v_{xx}=\frac{1}{\tau}(u-v)$.
\end{center}
Therefore, equation ($\ref{equation6}$) is equivalent to
\begin{equation}
	\left\{\begin{array}{llll}
		u_t+u_x-2avu_x-bu_{xxt}+\tau u_{xx}+ku_y=0,
		\\
	v_t-v_{xx}=\frac{1}{\tau}(u-v).
	\end{array}
	\right.\label{eq22}
\end{equation}
Letting $u(x,y,t)=\phi(\xi)$, $v(x,y,t)=\eta(\xi)$ and $\xi=x+y-ct$. Then $\phi$ and $\eta$ satisfy
\begin{equation}
	\left\{\begin{array}{llll}
	-c\phi'+\phi'-2a\eta\phi'+bc\phi'''+\tau\phi''+k\phi'=0,
		\\
	-c\eta'-\eta''=\frac{1}{\tau}(\phi-\eta),
	\end{array}
	\right.\label{equation39}
\end{equation}
where $ ’=\frac{d}{d\xi}$. Obviously,  system $(\ref{equation39})$ reduces to the following  singularly perturbed system
\begin{equation}
	\left\{\begin{array}{llll}
\phi'=\psi,
		\\
	\psi'=\omega,
	\\
	\omega'=\frac{1}{bc}[(c-k-1)\psi+2a\eta\psi-\tau\omega],
	\\
	\eta'=\widetilde\zeta,
	\\
	{\widetilde\zeta}'=-c\widetilde\zeta+\frac{1}{\tau}(\eta-\phi).
	\end{array}
	\right.\label{equation40}
\end{equation}
By $\epsilon=\sqrt{\tau}$ and $\zeta=\epsilon\widetilde\zeta$, we get
\begin{equation}
	\left\{\begin{array}{llll}
\phi'=\psi,
		\\
	\psi'=\omega,
	\\
	\omega'=\frac{1}{bc}[(c-k-1)\psi+2a\eta\psi-\epsilon^2\omega],
	\\
	\epsilon\eta'=\zeta,
	\\
	\epsilon\zeta'=-c\epsilon\zeta+\eta-\phi,
	\end{array}
	\right.\label{equation41}
\end{equation}
which is equivalent to
\begin{equation}
	\left\{\begin{array}{llll}
\dot\phi=\epsilon\psi,
		\\
	\dot\psi=\epsilon\omega,
	\\
	\dot\omega=\frac{\epsilon}{bc}[(c-k-1)\psi+2a\eta\psi-\epsilon^2\omega],
	\\
	\dot\eta=\zeta,
	\\
	\dot\zeta=-c\epsilon\zeta+\eta-\phi,
	\end{array}
	\right.\label{equation42}
\end{equation}
where $\dot{{}}=\frac{d}{dz}$ and $\xi=\epsilon z.$  If $\epsilon$ is set to zero in system ($\ref{equation42}$),
the  critical manifold is given by
\begin{center}
	$M_0=\big \{ (\phi,\psi,\omega,\eta,\zeta)\in {\mathbb{R}}^5 : \eta=\phi,\zeta=0\big \}.$
\end{center}

Similarly, we obtain the following linearized matrix of system ($\ref{equation42}$)
with $\epsilon=0$
\begin{equation*}
	\begin{pmatrix}
		0&0&0&0&0\\
		0&0&0&0&0\\
		0&0&0&0&0\\
		0&0&0&0&1\\
		-1&0&0&1&0
	\end{pmatrix}
\end{equation*}
where the number of zero eigenvalues of the matrix is equal to the dimension of $M_0$. According to the GSP theory of Fenichel \cite{ref14,ref31}, $M_0$ is a normally hyperbolic manifold.  Therefore, if $0<\epsilon \ll 1$, we obtain the three dimensional invariant manifold $M_\epsilon$, given by
\begin{center}
	$M_\epsilon=\big \{ (\phi,\psi,\omega,\eta,\zeta)\in {\mathbb{R}}^5 : \eta=\phi+g(\phi,\psi,\omega,\epsilon),\zeta=0+h(\phi,\psi,\omega,\epsilon)\big \},$
\end{center}
where  $g$ and $h$ are smooth  on a compact domain. Similarly, by comparing the coefficients of $\epsilon$ and its higher-order terms, we have
\begin{center}
	$g=(c\psi+\omega)\epsilon^2+o(\epsilon^2),\quad h=\epsilon\psi +o(\epsilon).$
\end{center}
 Restricted to $M_\epsilon$, the slow system is given by
\begin{equation}
	\left\{\begin{array}{llll}
		\phi'=\psi,
		\\
	\psi'=\omega,
	\\
	\omega'=\frac{1}{bc}[(c-k-1)\psi+2a\psi\phi+2a\psi(c\psi+\omega)\epsilon^2-\epsilon^2\omega]+o(\epsilon^2).
	\end{array}
	\right.\label{equation442}
\end{equation}
In order to obtain the standard form of singular perturbation in system ($\ref{equation442}$).
Let
\begin{equation*}
	\left\{\begin{array}{llll}
		\widetilde\phi=\phi,
		\\
		\widetilde\psi=\psi,
		\\
		\widetilde\omega=(c-k-1)\phi+a\phi^2-bc\omega,
	\end{array}
	\right.\label{eq37}
\end{equation*}
then we have
\begin{equation*}
	\left\{\begin{array}{llll}
		\widetilde\phi'=\widetilde\psi,
		\\
		\widetilde\psi'=\frac{1}{bc}[(c-k-1)\widetilde\phi+a\widetilde\phi^2-\widetilde\omega],
		\\
		\widetilde\omega'=-2ac\widetilde\psi^2\epsilon^2-(2a\widetilde\psi-1)\frac{\epsilon^2}{bc}[(c-k-1)\widetilde\phi+a\widetilde\phi^2-\widetilde\omega]+o(\epsilon^2).
	\end{array}
	\right.\label{eq1111}
\end{equation*}	
Again suppress the superscript tilde for the sake of simplicity
\begin{equation*}
	\left\{\begin{array}{llll}
		\phi'=\psi,
		\\
		\psi'=\frac{1}{bc}[(c-k-1)\phi+a\phi^2-\omega],
		\\
	\omega'=-2ac\psi^2\epsilon^2-(2a\psi-1)\frac{\epsilon^2}{bc}[(c-k-1)\phi+a\phi^2-\omega]+o(\epsilon^2),
	\end{array}
	\right.\label{eq1111}
\end{equation*}	
and in the form of involving $\tau$
\begin{equation}
	\left\{\begin{array}{llll}
		\phi'=\psi,
		\\
		\psi'=\frac{1}{bc}[(c-k-1)\phi+a\phi^2-\omega],
		\\
		\omega'=-2ac\psi^2\tau-(2a\psi-1)\frac{\tau}{bc}[(c-k-1)\phi+a\phi^2-\omega]+o(\tau).
	\end{array}
	\right.\label{equation43}
\end{equation}
\begin{theorem}
  There is a unique speed $c$, such that  system ($\ref{equation43}$) has a  homoclinic orbit which connects with two equilibrium $O_1$ and $O_2$ for sufficiently small $\tau>0$. It lies  $C^1$ $O(\tau)$ close to the undisturbed homoclinic orbit $\Gamma$ $(\Gamma_\tau \rightarrow \Gamma$ in $C^1$ norm as $\tau \rightarrow 0)$. Therefore,  equation ($\ref{equation6}$) with  special nonlocal distributed delay has a solitary wave solution $\Gamma_\tau$.
\end{theorem}
\indent We calculate the corresponding  Melnikov integral acrrording to  Lemma 4.1
\begin{equation}
	\begin{aligned}
		M(\phi,\psi)&=\frac{1}{bc}\int_{-\infty}^{+\infty}\psi \frac{ \partial \omega}{\partial \tau}d\xi\\
		&=\frac{1}{bc}\int_{-\infty}^{+\infty} \frac{ \partial \omega}{\partial \tau}d\phi\\
		&=\frac{1}{bc} \frac{ \partial \omega}{\partial \tau}\phi\big|_{-\infty}^{+\infty}-\frac{1}{bc}\int_{-\infty}^{+\infty}
		\phi \frac{d}{d\xi}(\frac{ \partial \omega}{\partial \tau})d\xi \\
		&=-\frac{1}{bc}\int_{-\infty}^{+\infty}-2ac\phi\sqrt{\frac{2a\phi^3-3(k+1-c)\phi^2}{3bc}}d\phi-\frac{1}{bc}\int_{-\infty}^{+\infty}2a\psi^2d\phi+\frac{1}{bc}\int_{-\infty}^{+\infty}\psi d\phi\\
		&=-\frac{2}{bc}\int_{0}^{\phi^*}(-2ac\phi+2a\sqrt{\frac{2a\phi^3-3(k+1-c)\phi^2}{3bc}}-1)\sqrt{\frac{2a\phi^3-3(k+1-c)\phi^2}{3bc}}d\phi\\
		&=\Delta(c),
	\end{aligned}
	\label{equation44}
\end{equation}	
where $\phi^*=\frac{3(k+1-c)}{2a}$ and  $M(\phi,\psi)$ is continuous in $c$. If $c\rightarrow k+1$, then the first term of the integral $(-2ac\phi+2a\sqrt{\frac{2a\phi^3-3(k+1-c)\phi^2}{3bc}}-1)$ is always negative, resulting
the corresponding value of $M(\phi,\psi)$ is positive. If $c\rightarrow +\infty$,   then $(-2ac\phi+2a\sqrt{\frac{2a\phi^3-3(k+1-c)\phi^2}{3bc}}-1)$ is positive, resulting the integral function $M(\phi,\psi)$ is negative. It follows that there exist a $c^*$ such that
\begin{center}
		$\Delta (c)=\Delta (c^*)=0, and$ $\frac{\partial \Delta}{\partial c}\big|_{c^*}\neq 0$.
	\end{center}
 Therefore, for each  $0<\tau\ll1$ , by using implicit function theorem, there exists a  $c=c^*(\tau)$ such that  $M(\phi,\psi)=0$. As a result, the distance function $N(c,\tau)$ has at least one zero solution for every $c>0$. That is, equation ($\ref{equation6}$) with  special nonlocal distributed delay has a solitary wave solution.

\begin{remark}
	The nonlocal delay term $f*u(x,y,t)$ and the perturbation term $\tau u_{xx}$ play  important roles in the existence of solitary wave solutions. The Melnikov integral function without the disturbance of viscous term $\tau u _{xx}$ is given by
	\begin{center}
		$M(\phi,\psi)=-\frac{2}{bc}\displaystyle\int_{o}^{\phi^*}(-2ac\phi+2a\sqrt{\frac{2a\phi^3-3(k+1-c)\phi^2}{3bc}})\sqrt{\frac{2a\phi^3-3(k+1-c)\phi^2}{3bc}}d\phi$.
	\end{center}
Notice that, for $a=-1$, when $c\rightarrow +\infty$, the first term of the integral $(2c\phi-2\sqrt{\frac{-2\phi^3-3(k+1-c)\phi^2}{3bc}})$ is always positive, resulting the integral function $M(\phi,\psi)$ is negative. When $c\rightarrow k+2$, the corresponding Melnikov integral is given by
\begin{equation*}
	\begin{aligned}
		M(\phi,\psi)
		&=\frac{-2}{\sqrt{3b^3(k+2)^3}}\big[\displaystyle\int_{0}^{\phi^*}2(k+2)\phi^2\cdot \sqrt{3-2\phi}d\phi-\displaystyle\int_{0}^{\phi^*}\frac{2}{\sqrt{3b(k+2)}}\phi^2 \cdot (3-2\phi)d\phi\big]\\
		&=\frac{-4}{\sqrt{3b^3(k+2)^3}}\big[\displaystyle\int_{0}^{\frac{3}{2}}(k+2)\phi^2\cdot \sqrt{3-2\phi}d\phi-\displaystyle\int_{0}^{\frac{3}{2}}\frac{1}{\sqrt{3b(k+2)}}\phi^2 \cdot (3-2\phi)d\phi\big]\\
		&=\frac{-4}{\sqrt{3b^3(k+2)^3}}\big[\displaystyle\int_{0}^{\frac{3}{2}}(k+2)\phi^2\cdot \sqrt{3-2\phi}d\phi-\frac{27}{32\sqrt{3b(k+2)}}\big]\\
		&=\frac{-4}{\sqrt{3b^3(k+2)^3}}\big[(k+2)\displaystyle\int_{\sqrt3}^{0}-(\frac{9}{4}t^2+\frac{3}{2}t^4-\frac{1}{4}t^6)dt-\frac{27}{32\sqrt{3b(k+2)}}\big]\\
		&=\frac{-4}{\sqrt{3b^3(k+2)^3}}\big[(k+2)\frac{18\sqrt3}{35}-\frac{27}{32\sqrt{3b(k+2)}}\big].
	\end{aligned}
	\label{equation45}
\end{equation*}	
 It is obvious that $M(\phi,\psi)$ is positive for bounded $k$ and sufficiently small $b$. Therefore, if equation ($\ref{equation6}$) is only affected by the nonlocal delay,  the solitary wave solutions still exist in certain cases. However, from Remark 3.3, we know that if equation ($\ref{equation6}$) is only affected  by local delay, the Melnikov function $M(\phi,\psi)$ has no zero solution. This difference is caused by the different actual background meanings of local and nonlocal delays (see in Conclusion). Without the nonlocal delay, $M(\phi,\psi)$ is calculated to be
	\begin{center}
		$M(\phi,\psi)=\frac{2}{bc}\displaystyle\int_{0}^{\phi^*}\sqrt{\frac{2a\phi^3-3(k+1-c)\phi^2}{3bc}}d\phi$,
	\end{center}
	which  keeps  positive for $c>0$. Similarly, this is consistent with the result of only being disturbed by the viscous term $\tau u_{xx}$ in Remark 3.3. Therefore, due to the combined effect of convolution $f*u$ and disturbance of viscous term $\tau u_{xx}$, a unique solitary wave solution is generated.
\end{remark}
\section{Numerical simulations}
In this section, we provide numerical simulations to validate the theoretical results presented earlier. Taking $b=1,a=-1$ and $k=-1$, we obtain\\
\indent \textbf{Case I.} with local distributed delay\\
\indent The corresponding Melnikov integral with local distributed delay is given by
\begin{equation*}
	\begin{aligned}
		M(\phi,\psi)&=\frac{2}{bc}\int_{0}^{\phi^*}(2ac\phi+1)\sqrt{\frac{2a\phi^3-3(k+1-c)\phi^2}{3bc}}d\phi \\
		&=\frac{2}{c}\int_{0}^{\frac{3}{2}c}(-2c\phi+1)\sqrt{\frac{-2\phi^3+3c\phi^2}{3c}}d\phi \\
		&=\frac{-4}{\sqrt{3c}}\int_{0}^{\frac{3}{2}c}\phi^2\sqrt{3c-2\phi}d\phi+\frac{2}{\sqrt{3c^3}}\int_{0}^{\frac{3}{2}c}\phi \sqrt{3c-2\phi}d\phi\\
		&=- \frac{72}{35}c^3+\frac{6}{5}c\\
		&=\Delta (c),
	\end{aligned}
	\label{equation41}
\end{equation*}
 where $c>0$ and the function diagram of $M(\phi,\psi)=\Delta (c)$   is obtained through numerical simulation (see Figure 2).

\begin{figure*}[h]
	\centering
	\includegraphics{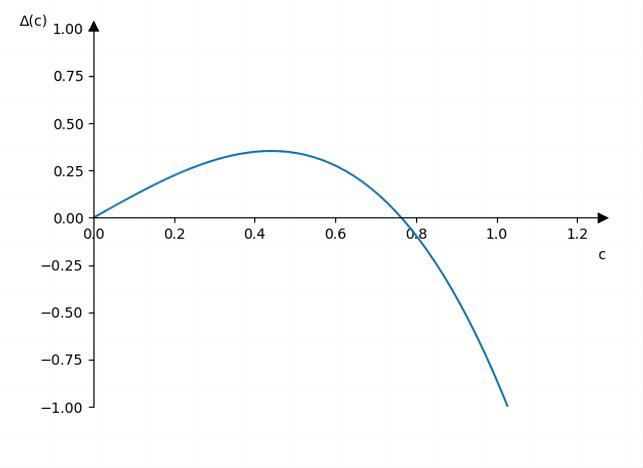}
	\caption{The case of local delay.}
	\label{fig2}
\end{figure*}

It is obvious that function $\Delta (c)=0$ has only one positive solution $c=c^*=\sqrt{\frac{7}{12}}$. By sample calculations, we obtain that

\begin{center}
		$\Delta (c)=\Delta (c^*)=0, and$ $\frac{\partial \Delta}{\partial c}\big|_{c^*}=-\frac{12}{5}<0$.
	\end{center}
	Therefore, for each  $0<\tau\ll1$ , by using implicit function theorem, there exists a  $c=c^*(\tau)$ such that  $M(\phi,\psi)=0$. As a result, the distance function $N(c,\tau)$ has at least one zero solution for every $c>0$. That is, equation ($\ref{equation6}$) with  special local distributed delay has a solitary wave solution.\\
\indent Moreover, through simple calculations, the Melnikov integral function without the local delay is given by
\begin{center}
	$M(\phi,\psi)=\frac{2}{bc}\displaystyle\int_{0}^{\phi^*}\sqrt{\frac{2a\phi^3-3(k+1-c)\phi^2}{3bc}}d\phi=\frac{6}{5}c$.
\end{center}
Notice that it keeps postive for $c>0$.  On the other hand, without disturbance of viscous term $\tau u_{xx}$, the corresponding Melnikov integral function is
\begin{center}
	$M(\phi,\psi)=\frac{2}{bc}\displaystyle\int_{0}^{\phi^*}2ac\phi\sqrt{\frac{2a\phi^3-3(k+1-c)\phi^2}{3bc}}d\phi=- \frac{72}{35}c^3$.
\end{center}
It keeps negative for $c>0$. \\
\indent \textbf{Case II.} with nonlocal distributed delay\\
\indent The corresponding Melnikov integral with nonlocal distributed delay given by
\begin{equation*}
	\begin{aligned}
		M(\phi,\psi)&=-\frac{2}{bc}\int_{0}^{\phi^*}(-2ac\phi+2a\sqrt{\frac{2a\phi^3-3(k+1-c)\phi^2}{3bc}}-1)\sqrt{\frac{2a\phi^3-3(k+1-c)\phi^2}{3bc}}d\phi \\
		&=\frac{2}{c}\int_{0}^{\frac{3}{2}c}(-2c\phi+1)\sqrt{\frac{-2\phi^3+3c\phi^2}{3c}}d\phi+\frac{4}{3c^2}\int_{0}^{\frac{3}{2}c}3c\phi^2-2\phi^3d\phi \\
		&=\frac{-4}{\sqrt{3c}}\int_{0}^{\frac{3}{2}c}\phi^2\sqrt{3c-2\phi}d\phi+\frac{2}{c\sqrt{3c}}\int_{0}^{\frac{3}{2}c}\phi \sqrt{3c-2\phi}d\phi+\frac{4}{3c^2}\int_{0}^{\frac{3}{2}c}3c\phi^2-2\phi^3d\phi \\
		&=- \frac{72}{35}c^3+\frac{6}{5}c+\frac{9}{8}c^2\\
		&=\Delta (c),
	\end{aligned}
	\label{equation42}
\end{equation*}
 where $c>0$ and the function diagram of $M(\phi,\psi)=\Delta (c)$   is obtained through numerical simulation (see Figure 3).\\
 
\begin{figure*}[h]
	\centering
	\includegraphics{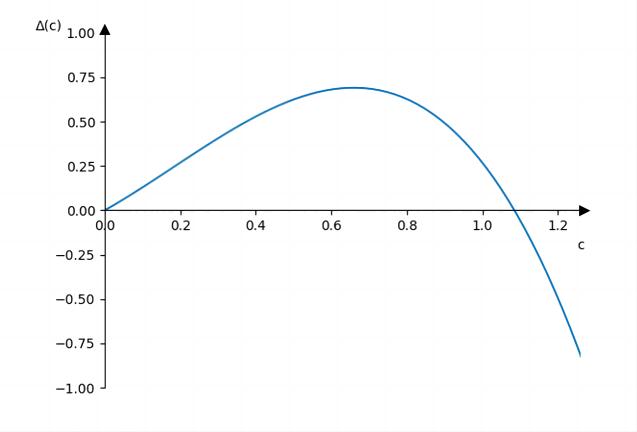}
	\caption{The case of nonlocal delay.}
	\label{fig3}
\end{figure*}
\indent It is obvious that function $\Delta (c)=0$ has only one positive solution $c=c^*$. By same calculations, we obtain that
\begin{center}
		$\Delta (c)=\Delta (c^*)=0, and$ $\frac{\partial \Delta}{\partial c}\big|_{c^*}< 0.$
	\end{center}
	 Therefore, for each  $0<\tau\ll1$ , by using implicit function theorem, there exists a  $c=c^*(\tau)$ such that  $M(\phi,\psi)=0$. As a result, the distance function $N(c,\tau)$ has at least one zero solution for every $c>0$. That is, equation ($\ref{equation6}$) with  special nonlocal distributed delay has a solitary wave solution.\\
 \indent The Melnikov integral function without the disturbance of viscous term $\tau u _{xx}$ is given by
 \begin{equation*}
	\begin{aligned}
		M(\phi,\psi)&=-\frac{2}{bc}\displaystyle\int_{0}^{\phi^*}(-2ac\phi+2a\sqrt{\frac{2a\phi^3-3(k+1-c)\phi^2}{3bc}})\sqrt{\frac{2a\phi^3-3(k+1-c)\phi^2}{3bc}}d\phi \\
		&=\frac{-4}{\sqrt{3c}}\int_{0}^{\frac{3}{2}c}\phi^2\sqrt{3c-2\phi}d\phi+\frac{4}{3c^2}\int_{0}^{\frac{3}{2}c}3c\phi^2-2\phi^3d\phi \\
		&=- \frac{72}{35}c^3+\frac{9}{8}c^2.
	\end{aligned}
	\label{equation42}
\end{equation*}
It is obvious that function $M(\phi,\psi)=\Delta (c)=0$ has only one positive solution $c=c^*=\frac{35}{64}$. On the other hand, without the nonlocal delay, $M(\phi,\psi)$ is calculated to be
\begin{center}
	$M(\phi,\psi)=\frac{2}{bc}\displaystyle\int_{0}^{\phi^*}\sqrt{\frac{2a\phi^3-3(k+1-c)\phi^2}{3bc}}d\phi=\frac{6}{5}c,$
\end{center}
which  keeps  positive for $c>0$.
\section{Conclusion}
In this article, we discuss the solitary wave solutions of the KP-BBM equation in different cases: $(i)$ without delay; $(ii)$ with local delay; $(iii)$ with nonlocal delay.  Firstly, we consider the unperturbed KP-BBM equation through method of dynamical system. Then by using the  GSP theory, especially the invariant manifold theory and  the Melnikov function, the solitary wave solutions of perturbed KP-BBM equation is proved. Then, from Remark 3.3 and Remark 3.5, we find some profound differences between local and nonlocal delays.  If equation ($\ref{equation6}$) is only affected  by local delay, the Melnikov function $M(\phi,\psi)$ has no zero solution. However, if equation ($\ref{equation6}$) is only perturbed by nonlocal delay,  the solitary wave solutions still exist in certain cases.
It is different from the result where the equation only receives the influence of local delay. Because the local delay is only a convolution about time, namely, it is only an integral of time $t$. While  the nonlocal delay is the integration of two variables, time $t$ and space $x$. It is precisely this difference that nonlocal delay does not affect the existence of solitary wave solutions to a certain extent. Moreover, we obtain that due to the combined effect of convolution $f*u$ and disturbance of viscous term $\tau u_{xx}$, a unique solitary wave solution is generated. This is the reason why they must exist simultaneously in the equation. Finally, we validate our results with numerical simulations.
\section{Fund Acknowledgement}
Yonghui Xia was supported by the
Zhejiang Provincial Natural Science Foundation of China (No. LZ24A010006). Hang Zheng was supported by the Young Scientists Fund of the National Natural Science
Foundation of China under Grant (No. 12301207), Natural Science Foundation of Fujian
Province under Grant (No. 2021J011148), Teacher and Student Scientific Team Fund
of Wuyi University (Grant No. 2020-SSTD-003).
\section{Author contributions}
We declare that the authors are ranked in alphabetic order of their names and all of them have the same
contributions to this paper.

\end{document}